\documentclass[psamsfonts,reqno,12pt]{amsart}
\usepackage{mathrsfs}
\usepackage{amsfonts}
\usepackage{stmaryrd}
\usepackage{amscd,amssymb,verbatim}
\usepackage{amssymb,amsmath,amsthm,epsfig}
\newtheorem{theorem}{Theorem}[section]
\newtheorem{lemma}[theorem]{Lemma}
\newtheorem{proposition}[theorem]{Proposition}

\theoremstyle{definition}

\newtheorem{remark}[theorem]{Remark}

\numberwithin{equation}{section}

\newcommand\N{\mathbb{N}}

\newcommand\Z{\mathbb{Z}}
\newcommand\F{\mathbb{F}}
\newcommand\M{\mathbb{M}}

\def\sideremark#1{\ifvmode\leavevmode\fi\vadjust{\vbox
to0pt{\vss \hbox to 0pt{\hskip\hsize\hskip1em
\vbox{\hsize2cm\tiny\raggedright\pretolerance10000
\noindent#1\hfill}\hss}\vbox to8pt{\vfil}\vss}}}

\begin{document}

\title[cb-frames for operator spaces]
{cb-frames
for  operator spaces}

\author{Rui Liu}
\address{School of Mathematical Sciences and LPMC, Nankai University, Tianjin 300071, China
}
\email{ruiliu@nankai.edu.cn
}
\author{Zhong-Jin Ruan}
\address{Department of Mathematics, University of Illinois at Urbana-Champaign, Urbana, Illinois 61801, USA}
\email{ruan@math.uiuc.edu}

\begin{abstract}
In this paper, we introduce the concept of cb-frames for operator
spaces. We show that there is a concrete cb-frame
for the reduced free group
$C^*$-algebra  $C_r^*(\F_2)$,  which is derived from
the infinite convex decomposition of the biorthogonal system
$(\lambda_s, \delta_s)_{s\in\F_2}$.
We show that, in general,   a separable operator space $X$ has  a cb-frame
 if and only if it has  the completely bounded approximation property if and only if
it is completely isomorphic to a completely
complemented subspace of an operator space with a cb-basis.
Therefore,  a discrete group  $\Gamma$ is weakly amenable if and only if  the reduced group
C*-algebra $C^*_r(\Gamma)$ has a cb-frame.
Finally, we show that, in contrast to Banach space case, there exists  a separable operator space,
which can not be completely isomorphic to a subspace of an operator space with a cb-basis.

\end{abstract}

\thanks{
The first author was partially supported by the NSFC 11101220 and 11201336. The
second author was partially supported by the Simons Foundation.}

\date {}

\keywords{completely bounded approximation property, cb-basis, cb-frame, operator space, reduced group $C^*$-algebra.}

\maketitle

\section{Introduction}

In \cite{JNRX}, Junge, Nielsen, Ruan and Xu introduced the notion of cb-basis for operator spaces.
Let us recall that a separable operator space $X$ has a \emph{cb-basis} if $X$ has a
 Schauder  basis $(e_n)$ and the natural projections
\begin{equation}\label{eq:6}
P_m\big(\sum_{n=1}^\infty \alpha_n e_n \big)=\sum_{n=1}^m
\alpha_n e_n
\end{equation} satisfy $\sup_m\|P_m\|_{cb}<\infty$.
It is shown in \cite {JNRX}   that every separable nuclear $C^*$-algebra has a
cb-basis. In particular,  the reduced group $C^*$-algebra $C_r^*(\Gamma)$ of  an amenable group $\Gamma$
has a cb-basis.


It is clear from the definition that if a separable  operator space $X$ has a cb-basis, then it must have the
\emph{completely bounded approximation property} (\emph{CBAP}), i.e. there exists a sequence  of
finite-rank maps $\Phi_k :  X\to X$ such that
$\sup_k \|\Phi_k\|_{cb} < \infty$  and $\Phi_k(x)\to x$ for every
$x\in X$.
It is natural to ask whether   CBAP implies cb-basis.
This is not true  for general operator spaces since there exist a separable Banach space $X$ which has
the \emph{bounded approximation property} (\emph{BAP}), but has no Schauder basis.
Then MIN(X), the space $X$ equipped with the MIN-operator space structure,  is an operator space with  CBAP, but  no cb-basis.
However the problem is still open for separable C*-algebras.  In particular, it is not known
whether $C^*_r(\Gamma)$ has a cb-basis if $\Gamma$ is a  weakly amenable discrete group.

In \cite{HL},  Han and Larson introduced the concept of frames as a compression of a basis.
It is a generalization of dual frame pairs from Hilbert spaces to Banach spaces. In \cite{CHL},
Casazza, Han and Larson showed  that a separable Banach space has the BAP
if and only if it has a frame.
Motivated by these results, we can consider cb-frames for  operator spaces.
Here is the definition.
Let $X$ be an operator space   and $X^*$ be its canonical operator  dual.
A sequence $(x_n,f_n) \subset X\times X^*$ is a
\emph{cb-frame} for $X$ if $(x_n,f_n)$ is a frame, i.e.
\[
x = \sum_{n=1}^\infty f_n(x) x_n =  \sum_{n=1}^\infty (x_n \otimes f_n)(x)
\]
 for all $x \in X$,  and the initial sums
\[
S_m(x) = \sum _{n=1}^m x_n \otimes f_n
\]
define completely bounded maps on $X$ with  $\sup_m \|S_m\|_{cb}<\infty$.

In section 2, we show that there is a natural  cb-frame for
$C_r^*(\F_2)$, which is derived from
the infinite convex decomposition of the biorthogonal system
 $(\lambda_s, \delta_s)_{s\in\F_2}$.
Here  $\lambda$ is the left regular
representation of $\F_2$  and $\delta_s \in B_r(\F_2) = C^*_r(\F_2)^*$
is the characteristic function at $s\in \F_2$.

In Section 3, we prove  some equivalent conditions for general  operator spaces.
We show in Theorem \ref {th:1} and Theorem \ref {th:2} that
a separable operator space $X$ has  a cb-frame  if and only if it has  the CBAP if and only if
it is completely isomorphic to a completely complemented subspace of an operator space with a cb-basis.
These are  natural operator space analogues of corresponding Banach space results by Casazza, Han and Larson \cite {CHL},
Johnson, Rosenthal and Zippin \cite{JRZ}, and Pe{\l}czy\'{n}ski  \cite{Pe}.
We also show in Remark \ref{re:1}  that the cb-basis  constructed in the proof
of Theorem \ref{th:2} is  minimal.

It is known from the  Banach-Mazur theorem   that any separable
Banach space can be isometrically embedded into $C[0,1]$,  which is a separable Banach space with a Schauder basis.
So it is natural to ask whether the corresponding result holds  for general operator spaces.
Using the  Hilbertian operator space $X_0$ constructed by Oikhberg and Ricard \cite {OR},
we show in Theorem \ref {th:4}  that this is false for   operator spaces.

\section{Cb-frame for $C_r^*(\F_2)$}

Let $\F_2$ be the free group of two generators and let $\lambda$ be the left regular representation of $\F_2$.
The reduced group $C^*$-algebra $C_r^*(\F_2)$
is defined to be the norm-closure of $\mathrm{span}\{\lambda_s|s\in\F_2\}$ in $B(\ell_2(\F_2))$.
It is well-known (see  \cite{Ha, HK})  that   $C^*_r(\F_2)$  has the completely contractive approximation property (CCAP).
%
The main result in  this section is the following theorem, in which we show that  $C_r^*(\F_2)$ has a cb-frame.

\begin{theorem}\label{pp:1}
There exists a cb-frame $(x_n,f_n)$  for  $C_r^*(\F_2)$,
which is derived from
the infinite convex decomposition of the biorthogonal system
$(\lambda_s, \delta_s)_{s\in\F_2}$.
More precisely,  there is a
surjective map $\varphi:\N\rightarrow \F_2$ and a sequence of
positive scalars $(a_n)$ satisfying
\begin{enumerate}
\item[(i)]
$x_n=\lambda_{\varphi(n)}, f_n=a_n \delta_{\varphi(n)}$ for all
$n\in\N$;
\item[(ii)] $\displaystyle\sum_{n\in \varphi^{-1}(s)} a_n=1$ for all $s\in \F_2$.
\end{enumerate}
\end{theorem}

The proof will be carried out in the following arguments and lemmas.
Let us first recall  from \cite{Ha} that the word length function $s\in\F_2\mapsto
|s|\in[0,\infty)$ is conditionally negative definite on $\F_2$, and thus by
Schoenburg's theorem the map   $\varphi_t(s)={\rm e}^{-t|s|}$ is a positive definite
function (with $\varphi_t(e)=1$) on $\F_2$ for any $t>0$.
This gives us a family of  unital completely positive maps
$(\Phi_t)_{t > 0}$ on $C_r^*(\F_2)$ such that
\begin{equation}
\label {F.2.1}
\Phi_t(\lambda_s)=\varphi_t(s)\lambda_s={\rm e}^{-t|s|}\lambda_s
\end{equation}
 for all  $s\in\F_2.$
It is clear from (\ref  {F.2.1}) that
$\lim_{t\rightarrow0^+}\|\Phi_t(\lambda_s)-\lambda_s\|=0$ for all  $s\in\F_2$.
Then  $(\Phi_t)_{t > 0}$ is a family of unital
completely positive maps converging to the identity on
$C_r^*(\F_2)$ in the point-norm topology.

Let $W_d=\{s\in\F_2:|s|=d\}$ be the set of all words in $\F_2$
with length $d$, and let $\chi_{W_d}$ be the characteristic
function on   $W_d$. Then
\[
P_d:\lambda_s\in
C_r^*(\F_2)\mapsto \lambda_s\chi_{W_d}(s)\in C_r^*(\F_2)
\]
is the finite-rank projection onto the subspace
$E_d= \mathrm{span}\{\lambda_s:|s|=d\}$. It
is known  (see \cite{Bu,Pi}) that this projection is completely bounded with
$\|P_d\|_{cb}\le 2d$.
For each $t > 0$ and $m \in {\mathbb N}$, we get a  completely bounded finite-rank map
\[
\Phi_{t,m}=\Phi_t(P_0+ \cdot \cdot \cdot +P_m)
\]
 from
$C_r^*(\F_2)$ onto $\sum _{d=0}^m E_d = \mathrm{span}\{\lambda_s:|s|\le m\}$. We can easily obtain the following lemma,
which is  known by experts in the fields.
 We include a calculation for the convenience of readers.

\begin{lemma} \label {P.2.2}
For each $t > 0$ and $m \in \N$,  we have
\[
 \|\Phi_t-\Phi_{t,m}\|_{cb}\le 2 \sum_{d=m+1}^\infty  {\rm e}^{-td}d \to 0  \qquad (\mbox{as} ~ m \to \infty).
 \]
\end{lemma}
\begin{proof}
Let $x=\sum_{d=0}^\infty\sum_{|s|=d}a(s)\otimes\lambda_s$ be an element in $B(\ell_2)\otimes C_r^*(\F_2)$
with finitely many $a(s)$ nonzero. Then we have
\begin{eqnarray*}
 \|(\mathrm{id}_{B(\ell_2)}\otimes\Phi_t &-& \mathrm{id}_{B(\ell_2)}\otimes\Phi_{t,m})(x)\| =
  \|\sum_{d=m+1}^\infty\sum_{|s|=d}a(s)\otimes\Phi_t(\lambda_s)\|  \\
   &\le&
   \sum_{d=m+1}^\infty\|\sum_{|s|=d}a(s)\otimes  {\rm e}^{-td}\lambda_s\|
   = \sum_{d=m+1}^\infty  {\rm e}^{-td}\|\sum_{|s|=d}a(s)\otimes\lambda_s\| \\
   &=&
   \sum_{d=m+1}^\infty  {\rm e}^{-td}\|(\mathrm{id}_{B(\ell_2)}\otimes P_d)(x)\|
   \le  \sum_{d=m+1}^\infty  {\rm e}^{-td}\|P_d\|_{cb}\|x\|.
\end{eqnarray*}
Since $\|P_d\|_{cb} \le 2d$, we can conclude  that
\[ \|\Phi_t-\Phi_{t,m}\|_{cb}\le \sum_{d=m+1}^\infty  {\rm e}^{-td}\|P_d\|_{cb}\le 2 \sum_{d=m+1}^\infty  {\rm e}^{-td}d.\]
Since  $\sum_{d=1}^\infty  {\rm e}^{-td}d$  is a convergent positive infinite series, its remainder part converges to $0$.
Therefore, we can conclude that
\[
\|\Phi_t-\Phi_{t,m}\|_{cb}\le 2 \sum_{d=m+1}^\infty  {\rm e}^{-td}d \to 0.
\]
\end{proof}

According to   Lemma \ref {P.2.2}, for each $t> 0$, we can find $m\in \N$ such that $\|\Phi_t - \Phi_{t, m}\|_{cb}$ is sufficiently small.
Therefore, we can carefully choose  a sequence of  (monotone decreasing) $t_k\to 0$  and a sequence of (monotone increasing)
$m_k\to \infty$ such  that
\begin{equation*}
\lim_{k\to\infty}  \|\Phi_{t_k} - \Phi_{t_k, m_k}\|_{cb} \le\lim_{k\to\infty}\sum_{d=m_k+1}^\infty  {\rm e}^{-t_k d} d
=\lim_{k\to\infty}\frac{ {\rm e}^{-t_k(m_k+2)}}{(1- {\rm e}^{-t_k})^2}+
\frac{(m_k+1)\, {\rm e}^{-t_k(m_k+1)}}{1- {\rm e}^{-t_k}}
=0.\end{equation*}
For example,  we can choose  $t_k=1/\sqrt{k}$ and $m_k=k$, and we get
\[ \lim_{k\to\infty}\sum_{d=k+1}^\infty  {\rm e}^{-d/\sqrt{k}} d
=  0. \]
In this case,
the finite-rank maps $\{\Phi_{1/\sqrt{k},k}\}_{k \in \N}$ (with
$\sup_k\|\Phi_{1/\sqrt{k},k}\|_{cb}< \infty$) converge to the identity map on $C_r^*(\F_2)$ in
the point-norm topology.

Let $\{\delta_s\}\in B_r(\F_2) = C^*_r(\F_2)^*$ be the biorthogonal functionals of
$\{\lambda_s\}$. 
For $k=1$, we set
\begin{eqnarray}\label {F.2.1a} \Psi_1(x)&=&\Phi_{1,1}(x)
=\sum_{|s|\le 1} {\rm e}^{-|s|}\delta_s(x)\lambda_s
= \sum_{|s|\le
1} \big( \lambda_s\otimes  {\rm e}^{-|s|} \delta_s \big)(x).
\end{eqnarray}
There are 5 terms in  (\ref {F.2.1a}).
If we  list these 5 terms by the index $1 \le j \le 5$ and use $y_{1, j}$  (respectively, $g_{1, j}$)
 for the corresponding $\lambda_s$
(respectively, $ {\rm e}^{-|s|} \delta_s$) in each term, we can write
\begin{equation}\label {F.2.1b}
\Psi_1 =  \sum_{|s|\le
1} \big( \lambda_s\otimes  {\rm e}^{-|s|} \delta_s \big) = \sum_{j=1}^{5} (y_{1,j}\otimes g_{1,j}).
\end{equation}
For $k\ge 2$, we set
\begin{eqnarray}\label {F.2.1c}
\Psi_{k}(x)&=&(\Phi_{1/\sqrt{k},k}-\Phi_{1/\sqrt{k-1},k-1})(x)\nonumber\\
&=&\sum_{|s|\le k-1}
( {\rm e}^{-|s|/\sqrt{k}}- {\rm e}^{-|s|/ \sqrt{k-1}})\delta_s(x)\lambda_s+\sum_{|s|=k}
 {\rm e}^{-|s|/\sqrt{k}}\delta_s(x)\lambda_s\nonumber\\
 &=&\sum_{|s|\le k-1}
\big(\lambda_s\otimes ( {\rm e}^{-|s|/\sqrt{k}}- {\rm e}^{-|s|/\sqrt{k-1}}) \delta_s\big)(x)+
\sum_{|s|=k} \big( \lambda_s\otimes  {\rm e}^{-|s|/\sqrt{k}} \delta_s\big)(x).
\end{eqnarray}
There are  $2 \cdot 3^k -1 = 1 + 4 + 4\cdot 3 + \cdots + 4 \cdot 3^{k-1} $ terms
in (\ref {F.2.1c}).
So if we list these terms by the index $1 \le j \le  2 \cdot 3^k -1$  and
we use  $y_{k, j}$ (respectively, $g_{k, j}$)  for the corresponding $\lambda_s$
(respectively,  $( {\rm e}^{-|s|/\sqrt{k}}- {\rm e}^{-|s|/\sqrt{k-1}}) \delta_s$
or  ${\rm e}^{-|s|/\sqrt{k}} \delta_s$) in each term, then we can write
\begin{equation}\label {F.2.2a}
\Psi_k = \sum_{|s|\le k-1}
\big(\lambda_s\otimes ( {\rm e}^{-|s|/\sqrt{k}}- {\rm e}^{-|s|/\sqrt{k-1}}) \delta_s\big)+
\sum_{|s|=k} \big( \lambda_s\otimes  {\rm e}^{-|s|/\sqrt{k}} \delta_s\big) =  \sum_{j=1}^{2\cdot 3^k-1} (y_{k,j}\otimes g_{k,j}).
\end{equation}
This  is a sequence of completely bounded maps on $C^*_r(\F_2)$ with
\begin{equation}\label {F.2.3a}
\|\Psi_k\|_{cb} = \|\Phi_{1/\sqrt{k},k}-\Phi_{1/\sqrt{k-1},k-1}\|_{cb} \le 2 \sup_k \|\Phi_{\frac{1}{\sqrt k}, k}\|_{cb}
\end{equation}
and for each  $x \in C^*_r(\F_2)$, we have
\begin{equation}\label {F.2.3b}
x  = \lim_{k\to \infty} \Phi_{1/\sqrt{k},k}(x) =  \Phi_{1, 1}(x) + \sum_{k=2}^\infty (\Phi_{1/\sqrt{k},k}(x) - \Phi_{1/\sqrt{k-1},k-1}(x) )
= \sum_{k=1}^\infty \Psi_k(x).
\end{equation}

Now to get a frame, we need to  further modify the terms in (\ref {F.2.1b}) and
(\ref {F.2.2a})  by defining  $x_{k,i}=y_{k,j}$ and $f_{k,i}=\frac {g_{k,j}}{(2\cdot3^k-1)^2}$ when
$i=p(2\cdot3^k-1)+j$ with $0\le p\le (2\cdot3^k-1)^2 -1$ and $1\le j\le 2\cdot3^k-1$.
In this case, we can write
\begin{equation}\label {F.2.7a}
\Psi_k = \sum_{j=1}^{2\cdot 3^k-1} (y_{k,j}\otimes g_{k,j}) = \sum_{i=1}^{(2\cdot3^k-1)^3}(x_{k,i}\otimes f_{k,i})
\end{equation}
 for  all $k\in \N$, and thus  for each $x\in C_r^*(\F_2)$, we have
\begin{eqnarray}\label{eq:frame}
x&=&\sum_{k=1}^\infty\Psi_k(x)=\sum_{k=1}^\infty\big ( \sum_{j=1}^{2\cdot3^k-1} (y_{k,j}\otimes g_{k,j})(x)\big ) \nonumber
= \sum_{k=1}^\infty \big ( \sum_{i=1}^{(2\cdot3^k-1)^3} (x_{k,i}\otimes f_{k,i})(x)\big ).
\end{eqnarray}

Now we need to set up an appropriate order to relate each pair $(k, i)$  with a positive integer $n$.
For $k=1$, we have $5^3$  terms related to $(1, i)$. So we simply let $n = i$.
For $k \ge 2$,  we let $n= \sum_{r=1}^{k-1} (2\cdot 3^r-1)^3 + i$ with   $1\le i \le (2 \cdot 3^k -1)^3$.

\begin{lemma} \label {P.2.3}
If we let  $x_n = x_{k, i}$ and $f_n = f_{k, i}$, then   $(x_n, f_n)$ is a frame for   $C^*_r(\F_2)$.
\end{lemma}
\begin{proof}
We need to show that  for every $x \in C^*_r(\F_2)$,  the infinite series $ \sum_{n=1}^\infty (x_n \otimes f_n)(x)$ converges to $x$.
Since  $x = \sum_{k=1}^\infty \Psi_k (x)$ is a convergent series in $C^*_r(\F_2)$,  for arbitrary $\epsilon>0$
there exists $k_0 \ge 2$ such that for any $k\ge k_0$
\begin{equation*}
 \| \sum_{r=k}^\infty \Psi_r(x)\|+\|\Psi_k(x)\|+\frac{\|x\|}{2\cdot 3^k-1}
< \epsilon.
\end{equation*}
For any $m>m_0= \sum_{r=1}^{k_0-1}(2\cdot3^r-1)^3$, we can write
$m= \sum_{r=1}^{k-1} (2\cdot 3^r-1)^3 + j$ for some  $k\ge k_0$ and $1\le j \le (2 \cdot 3^k  -1)^3$.
In this case, there exists $0\le p\le  (2\cdot3^k-1)^2 -1$ and $1\le q \le 2\cdot 3^k - 1$ such that
\begin{eqnarray*}
&&\|x- \sum_{n=1}^m (x_n\otimes f_n)(x) \| \\
&=&\|x-\sum_{r=1}^{k-1}\sum_{i =1}^{(2\cdot3^r-1)^3} (x_{r, i}\otimes f_{r, i})(x)-\sum_{t=1}^j (x_{k,t}\otimes f_{k,t})(x)\|\\
&=&\| \sum_{r=1}^\infty \Psi_r(x)-\sum_{r=1}^{k-1} \Psi_r(x) - \frac{p}{(2\cdot 3^k-1)^2}\Psi_k(x)
-\frac{1}{(2\cdot 3^k-1)^2}\sum_{i =1}^q (y_{k,i}\otimes g_{k,i})(x) \|\\
&\le& \| \sum_{r=k}^\infty \Psi_r(x)\|+\|\Psi_k(x)\|+\frac{\|x\|}{2\cdot 3^k-1} <  \epsilon.
\end{eqnarray*}
This shows  that $x=\sum_{n=1}^\infty (x_n \otimes f_n)(x)$ for every $x\in C^*_r(\F_2)$.
\end{proof}

\begin{lemma}\label {P.2.4}
The sequence $(x_n, f_n)$ is a cb-frame for $C^*_r(\F_2)$.
\end{lemma}
\begin{proof}
We need to show that the initial sums
$
S_m= \sum_{n=1} ^m (x_n \otimes f_n)
$
 are completely bounded maps on $C^*_r(\F_2)$ with $\sup \|S_m\|_{cb} < \infty$.
Let $m \in \N$ be a   positive integer. As we have seen in the proof of Lemma \ref {P.2.3} that we can write
$m= \sum_{r=1}^{k-1} (2\cdot 3^r-1)^3 + j$ for some  $k\in \N$ and $1\le j \le (2 \cdot 3^k -1)^3$, and
there exists $0\le p\le  (2\cdot3^k-1)^2 -1$ and $1\le q \le 2\cdot 3^k - 1$ such that
\begin{eqnarray*}
\|S_m\|_{cb} &=&  \| \sum_{n=1}^m x_n\otimes f_n \|_{cb}
= \|\sum_{r=1}^{k-1}\sum_{i=1}^{(2\cdot3^r-1)^3} (x_{r,i}\otimes f_{r,i})+\sum_{t=1}^j (x_{k,t}\otimes f_{k,t})\|_{cb}\\
&=&\| \sum_{r=1}^{k-1} \Psi_r+ \frac{p}{(2\cdot 3^k-1)^2}\Psi_k
 + \frac{1}{(2\cdot 3^k-1)^2}\sum_{i =1}^q (y_{k,i}\otimes g_{k,i}) \|_{cb}\\
&\le& \|\Phi_{\frac 1 {\sqrt {k-1}}, {k-1}} \|_{cb} + \frac{p}{(2\cdot 3^k-1)^2}\|\Psi_k\| _{cb} +\frac{2\cdot 3^k-1}{(2\cdot 3^k-1)^2} \\
&\le&  3\cdot \sup_k \|\Phi_{\frac 1 {\sqrt k}, k} \|_{cb}  +1.
\end{eqnarray*}

Finally, we let us recall from (\ref{F.2.1b}), (\ref {F.2.2a}) and (\ref{F.2.7a})  that we can write
\[
\Psi_1 = \sum_{n=1}^{5^3} (x_{n}\otimes f_{n})  = \sum_{i=1}^{5^3} (x_{1,i}\otimes f_{1,i})
=  \sum_{l=1}^{5^2}( \sum_{j=1}^{5}  (y^l_{1,j}\otimes \frac {g^l_{1,j}} {5^2}))
\]
with $y^l_{1,j} = y_{1,j} = \lambda_s$ and
$g^l_{1,j} = g_{1,j} =  \frac{ {\rm e}^{-|s|}}{5^2} \delta_s $ for corresponding $|s| \le 1$.
For $k \ge 2$, we can write
\begin{eqnarray*}
\Psi_k &=& \sum_{n= \sum_{r=1}^{k-1} (2\cdot 3^r-1)^3 + 1}^{ \sum_{r=1}^{k} (2\cdot 3^r-1)^3} (x_n \otimes f_n)
=  \sum_{i=1}^{(2\cdot3^k-1)^3}(x_{k,i}\otimes f_{k,i})
=\sum_{l=1}^{(2\cdot 3^k-1)^2} (\sum_{j=1}^{2\cdot 3^k-1} (y^l_{k,j}\otimes \frac {g^l_{k,j}}{(2\cdot 3^k-1)^2}))
\end{eqnarray*}
with $y^l_{k, j} = y_{k, j} = \lambda_s$ and $g^l_{k,j} = g_{k, j} = \frac {( {\rm e}^{-|s|/\sqrt{k}}- {\rm e}^{-|s|/\sqrt{k-1}})}{(2\cdot 3^k-1)^2} \delta_s $ or
$\frac {{\rm e}^{-|s|/\sqrt{k}} }{(2\cdot 3^k-1)^2}  \delta_s$ for corresponding $|s| \le k$.
So for each $s\in \F_2$, there are many positive integers $n\in \N$ such that $x_n = \lambda_s$.
We let $\varphi :  n\in \N \to s\in \F_2 $  be  the map such that $x_n = \lambda_s$, and let $a_n $ be the coefficient for the corresponding $\delta_s$.
Then it is easy to see that statement  (i) and (ii)  in theorem  hold true.
\end{proof}

\begin{remark}\label{re:2}
The cb-frame $(x_n,f_n)$ for $C_r^*(\F_2)$ in Theorem \ref{pp:1} is not unconditional,
that is, the series $x=\sum_{n=1}^\infty f_n(x)x_n$ does not converge unconditionally in norm for each $x$ in $C_r^*(\F_2)$ (see \cite{CHL}).
If it is unconditional,
then  for each $x \in  C_r^*(\F_2)$ the infinite series
\[
x= \sum_{n=1}^\infty (x_n \otimes f_n)(x) =  \sum_{n=1}^\infty a_n\delta_{\varphi(n)}(x)\lambda_{\varphi(n)}
\]
 converges unconditionally.
In this case,  we can rearrange its order such that
\begin{equation}\label {F.2.9}
x=\sum_{n=1}^\infty a_n\delta_{\varphi(n)}(x)\lambda_{\varphi(n)}
=\sum_{k=0}^\infty( \sum_{|s|=k}(\sum_{\varphi(n)=s}a_n) \delta_s(x)\lambda_s).
\end{equation}
We note that  for any $s \neq e$ in $\F_2$ the summation  $\sum_{\varphi(n)=s}a_n$ in  the last term of (\ref {F.2.9})
is a positive infinite series with $\sum_{\varphi(n)=s}a_n = 1$.
The last equality  makes sense since we can apply an $\varepsilon$-argument to replace such an infinite sum by a finite sum if necessary.
 So we can conclude from (\ref {F.2.9}) that
 \[
 x = \sum_{k=0}^\infty( \sum_{|s|=k}(\sum_{\varphi(n)=s}a_n) \delta_s(x)\lambda_s) =
 \sum_{k=0}^\infty( \sum_{|s|=k} \delta_s(x)\lambda_s)  = \sum_{k=0}^\infty P_k(x).
 \]
This  implies that for any $x\in  C({\mathbb T}) \cong C_r^*(\Z)\hookrightarrow C_r^*(\F_2)$
\[x = \sum_{k=0}^\infty P_k(x)=\lim_{K\to\infty}\sum_{k=0}^K P_k(x)=\lim_{K\to\infty}\sum_{j=-K}^K a_j z^j\] converges uniformly
in $C({\mathbb T})$.
This  is impossible. At this moment, it is not known whether there is any unconditional frame for $C_r^*(F_2)$.
\end{remark}

\section{Cb-frame, CBAP and complemented embedding property}

In this section, we prove that a separable operator space $X$ has  a cb-frame  if and only if it has  the CBAP if and only if
it is completely isomorphic to a completely complemented subspace of an operator space with a cb-basis.
We separate it to two results.

\begin{theorem}\label{th:1}
A separable  operator space $X$ has a cb-frame if and only if $X$ has the
CBAP.
\end{theorem}
\begin{proof} The ``only if'' part is obvious. We only need to  prove the ``if'' part.
The proof is motivated by Pe{\l}czy\'{n}ski's decomposition technique.
Suppose that   $X$ has  the CBAP.
Then  there is a sequence of finite-rank maps $(\Phi_k)$ on $X$ such that  $\sup_k
\|\Phi_k\|_{cb} \le K< \infty$ and $x=\lim_{k\to\infty} \Phi_k(x)$ for all $x\in X.$
Let $\Psi_1=\Phi_1 = \Phi_1 - \Phi_0$ (with $\Phi_0 = 0$) and $\Psi_k=\Phi_k-\Phi_{k-1}$ for $k\ge 2$.
Then $(\Psi_k)$ is a sequence of finite-rank maps on $X$ such that
$\sup_k \|\Psi_k\|_{cb} \le 2K < \infty$ and
$x=\sum_{k=1}^\infty \Psi_k(x)$
for all $x\in X$. Let $m(k)$ be the dimension of  $\Psi_k(X)$. It is known from Auerbach theorem that there exists a biorthogonal basis
 $(y_{k, j},  y^*_{k, j})_{1\le j \le m(k)}  $ for  $\Psi_k(X) $ such
 that $\|y_{k,j}\| \le 1$, $\|y^*_{k,j}\|\le 1$ and $\langle y_{k,i}, y^*_{k,j}\rangle = \delta_{ij}$.
For each $1\le j \le m(k)$, $g_{k,j} = y^*_{k,j} \circ \Psi_k $   is a bounded linear functional on $X$ such that $\|g_{k,j}\| \le 2K$.
We get  $(y_{k,j}, g_{k,j})_{1\le j\le m(k)}$ in $X \times X^*$ such that $\Psi_k=\sum_{j=1}^{m(k)} y_{k,j}\otimes g_{k,j}$
satisfying $\|y_{k,j}\|\le 1$ and $\|g_{k,j}\|\le 2K$.
As we have discussed before Lemma \ref {P.2.3},  we  define $x_{k, i} = y_{k,j}$ and
$f_{k, i} = \frac {g_{k,j}}{m(k)^2}$ when $i = p \cdot m(k) + j$ with $0\le p \le m(k)^2 -1$ and $1\le j \le m(k)$.
Then we can use a similar argument as that given in Lemma \ref {P.2.3} and Lemma \ref {P.2.4} to show
 that if we let  $x_n = x_{k, i}$ and $f_n = f_{k,i}$ when
$n= \sum_{r=0}^{k-1} m(r)^3 + i$ and   $1\le i \le m(k)^3$, then $(x_n, f_n)$ is a cb-frame for $X$.
\end{proof}

\begin{theorem}\label{th:2}
An operator space $X$ has a cb-frame if and only if $X$ is
completely isomorphic to a completely complemented subspace of an
operator space with a cc-basis.
\end{theorem}
\begin{proof} The ``if" part is obvious since every completely complemented
subspace of an operator space with the CBAP always has the CBAP.
We only need to prove the ``only if'' part.
Let $(x_i, f_i)\in X \times X^*$ be a cb-frame of $X$
with $\|x_i\|=1$ for all  $i\in \N$.
Let $c_{00}$ be the linear space of
all sequences of complex numbers with finitely many nonzeros, and
$(e_i)$ be the canonical basis of $c_{00}$.
For any $u\in \mathbb{M}_n(c_{00})\cong c_{00}(\M_n)$, there is a unique
linear expression $u=\sum u_i\otimes e_i$ with finitely many
$u_i\neq 0$ in $\M_n$. We define a norm $|\|\cdot|\|_n$ on $\mathbb{M}_n(c_{00})$ as follows:
\begin{equation}\label{eq:9}
|\|u|\|_n=\sup_{m\ge 1}\|\sum_{i=1}^m u_i\otimes x_i\|_n=
\max_{m\ge 1} \|\sum_{i=1}^m u_i\otimes
x_i\|_n<\infty.
\end{equation}
More precisely  for any $u,v\in\M_n(c_{00})$, we have
\begin{eqnarray*}
  |\|u+v|\|_n =
    \max_{m\ge 1} \|\sum_{i=1}^m(u_i+v_i)\otimes x_i \|_n
   \le \max_{m\ge 1}  ( \,  \|\sum_{i=1}^m u_i\otimes x_i \|_n+ \|\sum_{i=1}^m v_i\otimes x_i \|_n )
   \le |\|u\|_n+|\|v\|_n.
\end{eqnarray*}
If $|\|u|\|_n=0$, then $\|\sum_{i=1}^m u_i\otimes x_i\|_n=0$ for all $m\in\N$. By induction,
we can conclude $u=0$.
This shows that  each $|\|\cdot|\|_n$ is a norm on
$\M_n(c_{00})$. Moreover, $|\|\cdot|\|_n$ satisfies the following
properties:
\begin{enumerate}
\item[(N1)] \, $|\|u\otimes e_i|\|_n=\|u\otimes
x_i\|_n=\|u\|_{n}$ for all $u\in \M_n$ and $i\in\N$.
\item[(N2)] \, For any $u_i\in \M_n$ and $m\le l$,
we have $|\|\sum_{i=1}^m u_i\otimes e_i|\|_n\le |\|\sum_{i=1}^l
u_i\otimes e_i|\|_n.$
\item[(N3)] \, For any $x \in X$, $(\sum_{i=1}^n
f_i(x)e_i)_{n=1}^\infty$ is a Cauchy sequence in
$(c_{00},|\|\cdot|\|_1)$.
\end{enumerate}
Since (N1) and (N2) are obvious, we only need to prove (N3).
Since $(x_i, f_i)$ is a frame for $X$, for any $x \in X$, we have $x=\sum_{i=1}^\infty f_i(x)x_i$.
 Then for any $\epsilon>0$, there is $n_0$ such that,
for any $m > n\ge  n_0$, we have $\|\sum_{i=n+1}^m f_i(x)x_i\|<\epsilon.$
Now if we let $y_n=\sum_{i=1}^n f_i(x)e_i$ for each $n \in \N$,  we get
\begin{equation*}
    |\|y_m-y_n|\|_1= |\|\sum_{i=n+1}^m f_i(x)e_i |\|_1=\max_{n+1\le k\le m} \|\sum_{i=n+1}^k f_i(x)x_i \|<\epsilon.
\end{equation*}
This shows (N3).

Now we prove that $|\|\cdot|\|_n$ is actually an operator space matrix
norm. For  $u=\sum u_i\otimes e_i\in \mathbb{M}_n(c_{00}), w=\sum
w_i\otimes e_i\in \mathbb{M}_m(c_{00})$, and $\alpha\in
\M_{n},\beta\in \M_{n},$ we have
\begin{eqnarray}
|\|u\oplus w|\|_{n+m}&=& |\|\sum (u_i\oplus w_i)\otimes e_i |\|_{ n+m}
= \max_{l\ge 1} \|\sum_{i=1}^l (u_i\oplus w_i)\otimes
x_i \|_{n+m}  \nonumber \\
&=&\max_{l\ge 1} \| (\sum_{i=1}^l
u_i\otimes x_i )\oplus (\sum_{i=1}^l w_i\otimes
x_i ) \|_{n+m}  \nonumber \\
&=& \max_{l\ge 1} \max \Big\{ \|\sum_{i=1}^l u_i\otimes
x_i \|_n , \|\sum_{i=1}^l w_i\otimes
x_i \|_m  \Big\}\nonumber\\
&=&\max \Big\{\max_{l\ge 1} \|\sum_{i=1}^l u_i\otimes
x_i \|_n ,\max_{l\ge 1} \|\sum_{i=1}^l w_i\otimes
x_i \|_m \Big\} \nonumber \\
&=&  \max\{\,|\|u|\|_n,|\|w|\|_m\},
\end{eqnarray}
and
\begin{eqnarray}
|\|\alpha u \beta|\|_n&=& |\|\alpha (\sum u_i\otimes
e_i )\beta |\|_n= |\|\sum (\alpha u_i \beta)\otimes
e_i |\|_n\nonumber\\
&=&\max_{l\ge 1}  \|\sum_{i=1}^l (\alpha u_i \beta)\otimes
x_i \|_n =\max_{l}  \|\alpha (\sum_{i=1}^l
u_i\otimes x_i )\beta \|_n \nonumber\\
&\le&\max_{l\ge 1} \|\alpha\|_{{n}} \|\sum_{i=1}^l u_i\otimes
x_i \|_n \|\beta\|_{{n}} = \|\alpha\|_{{n}}|\|u|\|_n\|\beta\|_{{n}}.
\end{eqnarray}
Thus, by the abstract characterization theorem given in \cite{Ru}, this newly defined
 matricial norm
$\{|\|\cdot|\|_n\}$ determines an operator space structure on
$c_{00}$. We let  $Y= {{c_{00}}}^{-|\|\cdot |\|}$ denote the completion.
It is known from \cite[Fact 6.3] {FHHSPZ} that $(e_i)$ is a basis for $Y$.
According to (N2), $(e_i)$ is actually a cc-basis for $Y$ since for any
$u=\sum u_i\otimes e_i\in \mathbb{M}_n(c_{00}),$
the natural projections  $P_m$ satisfy 
%
\[
\|(P_m)_n(u)\|_n= |\|\sum_{i=1}^m u_i\otimes e_i |\|_n\le|\|u|\|_n.
\]
%

Now let us define  a linear map
\begin{equation}\label {F.Q.1}
Q :   \sum \alpha_i e_i \in c_{00}\to \sum \alpha_i x_i \in X.
\end{equation}
 For any $u = \sum u_i\otimes
e_i \in \M_n(c_{00})$, we get
 $$(Q)_n(u)=(Q)_n\big(\sum u_i\otimes
e_i\big)=\sum u_i\otimes x_i$$
and
$$\|(Q)_n(u)\|_n= \|\sum u_i\otimes x_i \| _n \le |\|u|\|_n$$
by (\ref{eq:9}).
Together with (N1), we get $\|Q\|_{cb}= 1$, and $Q$ can be
uniquely extended to the whole space $Y$.
On the other hand, we can define a linear map
\begin{equation}\label {F.T.1}
T:  x = \sum f_i(x)x_i \in X\to \sum f_i(x)e_i \in Y.
\end{equation}
 Then, by (N3), $T$ is well-defined. For any
$x=[x_{j,k}]\in \M_n(X)$, we have
\begin{eqnarray*}
(T)_n(x)&=&[\,T(x_{j,k})]=\big[\sum_{i=1}^\infty
f_i(x_{j,k})e_i\big]=\sum_{i=1}^\infty [ f_i(x_{j,k})]\otimes e_i
\end{eqnarray*}
and thus
\begin{eqnarray*}
|\|(T)_n(x) |\|_n &=&
|\|\sum_{i=1}^\infty [ f_i(x_{j,k})]\otimes e_i|\|_n
= \lim_{m\to\infty} |\|\sum_{i=1}^m [ f_i(x_{j,k})]\otimes e_i|\|_n \\
&=& \lim_{m\to\infty} \sup_{1\le l\le m} \Big\|\sum_{i=1}^l [ f_i(x_{j,k})]\otimes x_i\Big\|_n
 \le   \sup_{l\ge 1}  \|S_l \|_{cb}\|x\|_n.
\end{eqnarray*}
This shows $\|T\|_{cb}\le \sup_{l\ge 1} \left\|S_l\right\|_{cb}<\infty$. Moreover, for all $x\in X$,
$$QT(x)=Q\left(\sum f_i(x)e_i\right)=\sum f_i(x)x_i=x.$$ That is,
$QT=\mathrm{id}_X$. It follows that $Q$ is a surjection from $Y$ onto $X$ and that $T$
is injection from $X$ into $Y$. Since
$$\|x\|_n =\|(QT)_n(x)\|_n =\|(Q)_n(T)_n(x)\|_n \le|\|(T)_n(x)|\|_n$$
for all $x\in \M_n(X)$, $T$ is a complete isomorphism from $X$ onto
$T(X)\subset Y$. Moreover, $TQ$ is a completely bounded projection
from $Y$ onto $TQ(Y)=T(X)$. This completes the proof.
\end{proof}

\begin{remark}
Let $G$ be a  countable discrete group.
Then  by  Theorem \ref{th:1},  $G$ is weakly amenable, or equivalently  the reduced group C*-algebra $C^*_r(G)$ has the CBAP, if and only if $C^*_r(G)$ has a cb-frame.
Let $A(G)$ be the Fourier algebra of $G$.  Then $A(G)$  has a canonial operator space structure obtained by identifying $A(G)$ with the operator predual of the left group von Neumann algebra $VN(G)$.
It can be shown  that  $G$ is weakly amenable if and only if    $A(G)$  (respectively, $A(G)^{\rm op}$)
has the CBAP and thus has a cb-frame.
Then using the complex interpolation method, we can  show that if $G$ is weakly amenable, then  for each $1 < p < \infty$ the non-commutative $L_p$ -space
$L_p(VN(G)) = (VN(G), A(G)^{\rm op})_{\frac 1p}$ (with the canonical operator space structure intoduced by Pisier)
has the CBAP and thus has a cb-frame.
The converse statement is not necessarily true for non-commutative $L_p$ -spaces.
Indeed, it is known from \cite[Proposition 5.2]{JR} that if a countable residually  finite discrete group $G$ has  the AP,
then $L_p(VN(G))$ has a cb-basis.
This contains a very interesting class of groups.  For instance it includes many weakly amenable groups  such as ${\mathbb F}_n$,  $SL(2, {\mathbb Z})$ and $Sp(1, n)$, as well as some non-weakly amenable groups like ${\mathbb Z}^2 \rtimes SL(2, {\mathbb Z})$.

\end{remark}

\begin{remark}\label{re:1}
It is shown in Theorem \ref{th:2} that if $X$ has a cb-frame $(x_i,f_i)$,
then we can construct an operator space $Y$ with a cc-basis $(e_i)$ and completely
bounded maps $Q$ and $T$ satisfying (\ref {F.Q.1}) and (\ref {F.T.1}).
We note that such a  cb-basis $(e_i)$ in $Y$  is a  minimal choice.
Suppose that we have another operator space $Z$ with a cb-basis
$(z_i)$ and  completely bounded maps satisfying  (\ref {F.Q.1}) and (\ref {F.T.1}), i.e.

\[
Q_Z: \sum a_i z_i\in Z\to \sum a_i x_i\in X
~\mbox{and} ~ T_Z: x= \sum f_i(x)x_i \in X \to \sum f_i(x)z_i\in Z.
\]
Then
for any $n\in\N$ and $u_i\in \M_n$ we have
\begin{eqnarray*}
|\|\sum u_i\otimes
e_i |\|_n &=&\max_{m\ge1} \|\sum_{i=1}^m u_i\otimes
x_i \|_n =\max_{m\ge1} \|\sum_{i=1}^m u_i\otimes
Q_Z(z_i) \|_n \\&\le&\|Q_Z\|_{cb}\max_{m\ge1} \|\sum_{i=1}^m
u_i\otimes z_i \|_n \le
\|Q_Z\|_{cb}K_Z \|\sum u_i\otimes z_i \|_n,
\end{eqnarray*}
where $K_Z = \sup _m \|P^Z_m\|_{cb} $ is the cb-constant of the cb-basis $\{z_i\}$.
This shows that  there is a constant $K = \|Q_Z\|_{cb} K_Z >0$ such that
\[ |\|\sum u_i\otimes
e_i |\|_{n}\le K \|\sum u_i\otimes z_i \|_{n}.\]
This shows that the cb-basis $(e_i)$ in $Y$   is
completely dominated by  such  a cb-basis $(z_i)$ in $Z$.
Therefore, the cb-basis $(e_i)$ in $Y$ constructed in Theorem \ref{th:2}
is  a minimal choice associated to the cb-frame $(x_i,f_i)$ in $X$.
\end{remark}

As in Banach space theory, we can also consider the unconditional case.
We say that an unconditional basis $(e_i)$ of an operator space $Y$ is \emph{completely
unconditional}  if
\begin{equation}
\sup_{E\subset\N, \, \#E<\infty}\|P_E\|_{cb}<\infty,
\end{equation}
where $E$ is a finite subset of $\N$ and $P_E$ is the natural  projection defined by
$P_E(\sum \alpha_ie_i)=\sum_{i\in E} \alpha_ie_i$.
We say that $(e_i)$ is \emph{completely 1-unconditional}  if $\sup_{E\subset\N, \, \#E<\infty}\|P_E\|_{cb}=1$,
or equivalently, $\|P_E\|_{cb}=1$ for any finite subset $E\subset\N$
(see  \cite{Oi}).
An unconditional frame
$(x_i,f_i)$ of an operator space $X$ (see Remark \ref{re:2}) is \emph{completely unconditional} if
\begin{equation}
\sup_{E\subset\N, \, \#E<\infty}\|S_E\|_{cb}<\infty,
\end{equation}
where $S_E$ is the natural partial sum map defined by
$S_E(x)=\sum_{i\in E} f_i(x)x_i$.

The following result is the unconditional version of Theorem \ref{th:2}, which
is also the operator space version of  Theorem 3.6 in \cite  {CHL}
\begin{theorem}\label{th:5}
An operator space $X$ has a completely unconditional cb-frame if and
only if $X$ is completely isomorphic to a completely complemented
subspace of an operator space with a completely 1-unconditional
cb-basis.
\end{theorem}
\begin{proof} Because the whole proof is similar to that of Theorem
\ref{th:2}, we only give out the construction of the matrix norm.
Without loss of generality, we assume that $(x_i,f_i)$ is a completely unconditional cb-frame
of $X$ with $\|x_i\|=1$ for all $i\in \N$.
Let   $(e_i)$  be  the unit vector basis of $c_{00}$.
For any $u\in \mathbb{M}_n(c_{00})$, there is a unique linear expression
$u=\sum u_i\otimes e_i$ where $u_i\in\M_n$ we can define
\begin{equation*}
|\|u|\|_n=|\|\sum u_i\otimes e_i |\|_n=\max_{E\subset\N, \,
\#E<\infty}\big\|\sum_{i\in E} u_i\otimes x_i\big\|_{n}.
\end{equation*}
It is easy to verify  that $|\|\cdot |\|_n$ is a norm on
$\mathbb{M}_n(c_{00})$ for each $n\in\N$, which satisfy
\begin{enumerate}
\item[$\mathbf{(U1)}$ ] $|\|u\otimes e_i|\|_n=\|u\otimes
x_i\|_{n}=\|u\|_{n}\|x_i\|$ for all $u\in \M_n$ and
$i\in\N$,
\item[$\mathbf{(U2)}$ ] For any $u_i\in \M_n$ and $E\subset\N$ with
$\#E<\infty$, we have $$|\|\sum_{i\in E} u_i\otimes
e_i |\|_n\le |\|\sum u_i\otimes e_i  |\|_n,$$
\item[$\mathbf{(U3)}$ ] For any $x \in X$, $(\sum_{i=1}^n
f_i(x)e_i)_{n=1}^\infty$ is a Cauchy sequence in
$(c_{00}, |\|\cdot |\|_1)$.
\end{enumerate}
Then the rest of proof is similar to that given for  Theorem \ref{th:2}.
\end{proof}


\section {Subspaces of  operator spaces with a cb-basis}

The  Banach-Mazur theorem shows   that every separable
Banach space can be isometrically embedded into $C[0,1]$.
Therefore, every  separable Banach space  can be isometrically embedded into a Banach space with a Schauder basis.
However, the corresponding result  is not  true for operator spaces.
We   show in  Theorem \ref {th:4}  that the  Hilbertian operator space $X_0$ constructed by Oikhberg and Ricard \cite {OR}
can not be completely isomorphic to any  subspace of an operator space with a cb-basis (or with a cb-frame).

Let us  first develop some notations and preliminary results.  Most of these  are motivated from  Banach space theory.
Let $X$ be an operator space.  A sequence $(x_i)\subset X$ is called a \textit{cb-basic sequence} in $X$ if $(x_i)$ is a
cb-basis of $\overline{\mathrm{span}}(x_i)$, the norm closure of $\mathrm{span} (x_i)$ in $X$.
Thus, any cb-basis is a cb-basic sequence.
Fix a sequence $(x_i)\subset X$, then
 a sequence $(y_i)\subset X$ is said to be a \emph{block sequence} of
 $(x_i)$ if there exists a sequence of positive integers $m_1<m_2<m_3<\cdot\cdot\cdot$
 such that $y_n\in \mathrm{span}(x_i)_{i=m_n}^{m_{n+1}-1}$ for all $n\in\N.$

\begin{proposition}\label{pp:3}
Let $X$ be an operator space.   Then a sequence
$(e_j)$ in $X$  is a cb-basic sequence if and only if there is a constant
$K\ge 1$ such that for any $u_i\in \M_n$ and $m\le l$, we have
\begin{equation}\label{eq:11} \|\sum_{j=1}^m u_j\otimes e_j \|_n\le
K \|\sum_{j=1}^l u_j\otimes e_j \|_n.
\end{equation}
As a consequence, any block sequence of a cb-basic sequence is a cb-basic sequence.
\end{proposition}
\begin{proof} The necessity is clear. For sufficiency, we only need to prove that
$(e_j)$ is a cb-basis of $\overline{\mathrm{span}}(e_j)$. By
\cite[Proposition 6.13]{FHHSPZ}, $(e_j)$ is a basis of $\overline{\mathrm{span}}(e_j)$.
For any $m,n\in\N$ and $(u_i)\subset\M_n$ with finitely many $u_i \neq 0$, it follows from (\ref{eq:11}) that
\[ \|(P_m)_n (\sum_j u_j\otimes e_j ) \|_n= \|\sum_{j=1}^m u_j\otimes e_j \|_n\le K \|\sum_j u_j\otimes e_j \|_n.\]
This shows that $\sup_m\|P_m\|_{cb}\le K$, and thus, $(e_j)$ is a cb-basic sequence in $X$.
\end{proof}

\begin{proposition}\label{pp:2} Let $X$ be an operator space with a cb-basis $(x_i)$. If \,$Y$
is an infinite-dimensional closed subspace of $X$, then $Y$ contains a cb-basic sequence.

\end{proposition}
\begin{proof} Let $K$ be the cb-basic constant of $(x_i)$. Given $p\in\N$, let $W_p$ be the finite-codimensional subspace of $X$ defined by
$$W_p=\Big\{x\in X:x=\sum_{i=p+1}^\infty a_i x_i\Big\}=\overline{\mathrm{span}(}x_i)_{i>p}.$$
Then $W_p\cap Y\neq\emptyset$, so there is $y\in Y\cap W_p$ with $\|y\|=1$.
First, choose an arbitrary $y_1\in\sum_{i=1}^\infty a_i^1 x_i\in Y$ with $\|y_1\|=1.$ Find $p_1\in\N$ such that for $u_1=\sum_{i=1}^{p_1} a_i^1 x_i\in X$ we have $\|y_1-u_1\|<\frac{1}{4K}.$ Choose $y_2=\sum_{i=p_1+1}^\infty a_i^2 x_i\in Y\cap W_{p_1}$ with $\|y_2\|=1$, and fix $p_2\in\N$ such that for $u_2=\sum_{i=p_1+1}^{p_2}a_i^2 x_i$ we have $\|y_2-u_2\|\le\frac{1}{2\cdot 2^2K}$. Using an induction procedure we obtain a block cb-basic sequence $(u_j)$ of $(x_i)$. Let $(u_j^*)$ be the biorthogonal functionals of $(u_j)$. Since $$\sum_{j=1}^\infty\|u_j^*\|\cdot\|y_j-u_j\|<\sum_{j=1}^\infty 2K\cdot\frac{1}{2\cdot 2^jK}<1,$$ by Lemma 2.13.2 in \cite{Pi},
there is a complete isomorphism $R:X\to X$ such that $R(u_j)=y_j$ for $j\in\N$.
Then for any $v_j\in\M_n$ and $m\le l$,
\begin{eqnarray*}
\big\|\sum_{j=1}^m v_j\otimes y_j\big\|_n &\le&
\|R\|_{cb}\big\|\sum_{j=1}^m v_j\otimes u_j\big\|_n
\le
K\|R\|_{cb}\big\|\sum_{j=1}^l v_j\otimes u_j\big\|_n \\
 &\le& K\|R\|_{cb}\|R^{-1}\|_{cb}\big\|\sum_{j=1}^l v_j\otimes y_j\big\|_n.
\end{eqnarray*}
Thus by Proposition \ref{pp:3}, $(y_j)\subset Y$ is a cb-basic sequence.
\end{proof}

\begin{proposition}\label{pp:4}
Let $X$ be a separable operator space. Then  following are equivalent:
\begin{enumerate}
  \item[(i)] $X$ is completely isomorphic to a subspace of an operator space with a cb-basis;
  \item[(ii)] $X$ is completely isomorphic to a subspace of an operator space with a cb-frame
   \item[(iii)] $X$ is completely isomorphic to a subspace of a (not necessarily separable) operator space with the CBAP.
 \end{enumerate}
\end{proposition}
\begin{proof}
It is obvious that (i)$\Rightarrow$(ii)    and (i) $\Rightarrow$ (iii).

(ii)  $\Rightarrow$ (i) Suppose that $X$  is completely isomorphic to a subspace of an operator space $Y$ with a cb-frame.
 By Theorem \ref{th:2},
$Y$ is completely isomorphic to a completely complemented subspace of an operator space $Z$ with a cb-basis.
Then  we can conclude from  the following completely isomorphic embeddings
\[ X \hookrightarrow Y \hookrightarrow Z.\]
that  $X$ is completely isomorphic to a subspace of $Z$ with a cb-basis.

(iii)  $\Rightarrow$ (i)  If   the operator space with the CBAP is separable, then it has a cb-frame by  Theorem \ref {th:1}
and thus we can get the result by (i) $\Leftrightarrow$ (ii).
If  the operator space with the CBAP is non-separable, we can obtain the result by   the following  lemma.
\end{proof}

\begin{lemma}\label {P.4.3}
If $X$ is an operator space with the $\lambda$-CBAP,
then, for every separable subspace $Y$ of $X$, there is a separable subspace $\tilde{Y}$ of $X$ which contains $Y$ and has the $\lambda$-CBAP.
\end{lemma}
\begin{proof}
Let $Y$ be a separable subspace of $X$.
There is a sequence $(y_n)$ which is dense in $Y$. For each $n\in\N$, put $Y_n=\mathrm{span}\{y_k\}_{k=1}^n$.
Then $\{Y_n\}$ is an increasing sequence of finite-dimensional subspaces of $Y$ whose union is dense in $Y$.
Since $B_{Y_1}$ is compact, there is a finite-rank map $F_1$ on $X$ with $\|F_1\|_{cb}\le\lambda$ and
$\|F_1(y)-y\|<1$ for all $y\in B_{Y_1}.$ Let $\tilde{Y_2}=\mathrm{span}\,Y_2\cup F_1(X)$. Then
there is a finite-rank map $F_2$ on $X$ with $\|F_2\|_{cb}\le\lambda$ and
$\|F_2(y)-y\|<1/2$ for all $y\in B_{\tilde{Y_2}}.$ Set $\tilde{Y_3}=\mathrm{span}\,Y_3\cup F_1(X)\cup F_2(X)$. Then
there is a finite-rank map $F_3$ on $X$ with $\|F_3\|_{cb}\le\lambda$ and
$\|F_3(y)-y\|<1/3$ for all $y\in B_{\tilde{Y_3}}.$ Applying the procedure to $\tilde{Y_3}$ and $F_3$ gives
$\tilde{Y_4}$ and $F_4$ and so on by induction.
Let \[\tilde{Y_n}=\mathrm{span}(Y_n\cup \cup_{k=1}^{n-1}F_k(X)).\] Then
there is a finite-rank map $F_n$ on $X$ with \[\|F_n\|_{cb}\le\lambda, \ \
\|F_n(y)-y\|<\frac{1}{n}, \, \forall \,y\in B_{\tilde{Y_n}}.\]
So $\{\tilde{Y_n}\}$ is an increasing sequence of
finite-dimensional subspaces of $X$, and define $\tilde{Y}$ to be the closure of $\cup_{n=2}^\infty \tilde{Y_n}$,
which is a separable subspace of $X$ with $Y\subseteq \tilde{Y}$ and $F_n(\tilde{Y})\subseteq \tilde{Y}$ for all $n\in\N$.
It is easy to prove that $\lim_{n\to\infty} F_n(y)=y$ for all $y\in\tilde{Y}$.
Then $\tilde{Y}$ has the $\lambda$-CBAP.
\end{proof}

Now we are ready to prove our main result in this section.
In  \cite {OR},  Oikhberg and Ricard constructed a separable  Hilbertian operator space $X_0$, which is isometrically isomorphic  to $\ell_2(\N)$,
 but every infinite-dimensional closed subspace of $X_0$ fails to have the \emph{operator space approximation
property} (\emph{OAP}).

\begin{theorem}\label{th:4}
The Oikhberg-Ricard space $X_0$ can not be completely isomorphic to any    subspace of an operator space with a cb-basis (or with a cb-frame).
\end{theorem}
\begin{proof}
Suppose that $X_0$ is completely isomorphic to an (infinite-dimensional)    subspace $Y$ of an  operator space $Z$ with a cb-basis (or with a cb-frame).
It  is known from Proposition \ref {pp:2}  and Proposition \ref {pp:4} that  $Y$ must have a cb-basic sequence.
It follows that   $X_0$ has a cb-basic sequence $(e_j)$.
Then $\overline{\mathrm{span}}(e_j)$ is an infinite-dimensional closed subspace of $X_0$ having the CBAP.
This contradicts to Oikhberg-Ricard's result since    CBAP implies OAP (see \cite{ER}).
\end{proof}

\begin{remark}
Finally we note that  every 1-exact separable operator space is completely isomorphic to a subspace of an operator space with a cb-basis.
To see this let us recall from  \cite{EOR} that if $X$ is a  1-exact separable operator space,  then it is  completely isometric to a subspace of
a $1$-nuclear separable operator space $Y$.  Then by Propoosition \ref {pp:4}, $X$   is completely isomorphic to a subspace of an operator space with a cb-basis.
\end{remark}

\vskip1cm
%
%
\vskip1cm
%


\begin{thebibliography}{1}

%
%

\bibitem{Bu} A. Buchholz, Norm of convolution by operator-valued functions on free groups, Proc. Amer.
Math. Soc. 127 (1999) 1671-1682.


%
%
%
%



\bibitem{CHL} P.G. Casazza, D. Han and D.R. Larson,
Frames for Banach spaces, Contemp. Math. 247 (1999) 149-182.


%
%

\bibitem{EOR} E. Effros, N. Ozawa and Z. Ruan, On injectivity and nuclearity for operator
spaces, Duke Math. J. 110 (2001) 489-521.

\bibitem{ER} E. Effros and  Z. Ruan, Operator spaces, London
Mathematical Society Monographs New Series, 23, Oxford Science
Publications, 2000.


\bibitem{FHHSPZ} M. Fabian, P. Habala, P. H\'{a}jek, V.M.
Santaluc\'{\i}a, J. Pelant and V. Zizler, Functional analysis and
infinite-dimensional geometry, CMS books in mathematics 8, Springer,
2001.

%

\bibitem{Ha} U. Haagerup, An example of a nonnuclear C$^*$-algebra, which has the
metric approximation property, Invent. Math. 50 (1978/79), no. 3,
279-293.


\bibitem{HK} U. Haagerup and J. Kraus, Approximation properties for group $C^*$-algebras and
group von Neumann algebras, Tran. Amer. Math. Soc. 344 (1994), no. 2, 667-699.

\bibitem{HL}  D. Han and D. R. Larson, Frames, bases and group
representations, Mem. Amer. Math. Soc. 697, 2000.



\bibitem{JRZ} W.B. Johnson, H.P. Rosenthal and M. Zippin, On bases,
finite dimensional decompositions and weaker structures in Banach spaces, Israel J. Math.
9 (1971), 488-506.

\bibitem{JNRX} M. Junge, N. Nielsen, Z. Ruan and Q. Xu, $\mathscr{C}\mathcal
{O}\mathscr{L}_p$ spaces---the local structure of non-commutative
$L_p$ spaces, Advances in Math. 187 (2004),  257-319.


\bibitem{JR} M. Junge and  Z. Ruan, Approximation properties for noncommutative $L_p$-spaces associated
with discrete groups, Duke Math. J. 117 (2003), 313-341.


%


%


%

\bibitem{Oi} T. Oikhberg, Operator spaces with complete bases, lacking completely unconditional bases, Houston J. Math. 32 (2006) 551-561.

\bibitem{OR} T. Oikhberg and E. Ricard, Operator spaces with few completely bounded maps, Math. Ann. 328 (2004), 229-259.

%

\bibitem{Pe} A. Pe{\l}czy\'{n}ski, Any separable Banach space with the bounded approximation
property is a complemented subspace of a Banach space with a basis, Studia Math. 40 (1971), 239-243.

\bibitem{Pi} G. Pisier, Introduction to operator space theory,
London Mathematical Society Lecture Note Series, 294, Cambridge
University Press, 2003.


\bibitem{Ru} Z.  Ruan, Subspaces of $C^*$-algebras, J. Func.
Anal. 76 (1988) 217-230.



\end{thebibliography}
\end{document}